\documentclass[12pt]{amsart}

\usepackage{amssymb}
\usepackage{bm}
\usepackage{graphicx}
\usepackage{psfrag}
\usepackage{color}
\usepackage{url}
\usepackage{algpseudocode}
\usepackage{fancyhdr}
\usepackage{xy}
\input xy
\xyoption{all}

\newtheorem{theorem}{Theorem}[section]
\newtheorem{prop}[theorem]{Proposition}
\newtheorem{lemma}[theorem]{Lemma}
\newtheorem{cor}[theorem]{Corollary}
\theoremstyle{definition}
\newtheorem{definition}[theorem]{Definition}
\newtheorem{example}[theorem]{Example}
\numberwithin{equation}{section}

\newcommand\nn{\mathbb{N}}

\newcommand\zz{\mathbb{Z}}

\begin{document}

	\mbox{}
	\title{On Delta Sets and their Realizable Subsets \\ in Krull Monoids with Cyclic Class Groups}
	\author{Scott T. Chapman}
	\address{Department of Mathematics\\Sam Houston State University\\Huntsville, TX 77341}
	\email{scott.chapman@shsu.edu}
	
	\author{Felix Gotti}
	\address{Department of Mathematics\\University of Florida\\Gainesville, FL 32611}
	\email{felixgotti@ufl.edu}
	
	\author{Roberto Pelayo}
	\address{Department of Mathematics\\University of Hawai'i at Hilo\\Hilo, HI 96720}
	\email{robertop@hawaii.edu}
	
	\date{\today}
	
	\begin{abstract}
		Let $M$ be a commutative cancellative monoid. The set $\Delta(M)$, which consists of all positive integers which are distances between consecutive factorization lengths of elements in $M$, is a widely studied object in the theory of nonunique factorizations. If $M$ is a Krull monoid with cyclic class group of order $n \ge 3$, then it is well-known that $\Delta(M) \subseteq \{1, \dots, n-2\}$. Moreover, equality holds for this containment when each class contains a prime divisor from $M$. In this note, we consider the question of determining which subsets of $\{1, \dots, n-2\}$ occur as the delta set of an individual element from $M$. We first prove for $x \in M$ that if $n - 2 \in \Delta(x)$, then $\Delta(x) = \{n-2\}$ (i.e., not all subsets
		of $\{1,\dots, n-2\}$ can be realized as delta sets of individual elements). We close by proving an Archimedean-type property for delta sets from Krull monoids with finite cyclic class group: for every natural number m, there exist a Krull monoid $M$ with finite cyclic class group such that $M$ has an element $x$ with $|\Delta(x)|  \ge m$.
	\end{abstract}
	
	\maketitle
	
	\section{Introduction} \label{sec:introduction}
	
	The arithmetic of Krull monoids is a well-studied area in the theory of nonunique factorizations. The interested reader can find a good summary of their known arithmetic properties in the monograph \cite[Chapter~6]{GH06}. We focus here on Theorem~6.7.1 of \cite{GH06}, where the authors show that
	\[
		\Delta(M) = \{1, \dots, n-2\}
	\]
	for $M$ a Krull monoid with cyclic class group of order $n \ge 3$ where each class contains a prime divisor of $M$. Here $\Delta(M)$ represents the set of all positive integers which are distances between consecutive factorization lengths of elements in $M$.
	
	We ask in this note a question related to the above equality that is seemingly unasked in the literature: Which subsets $T \subseteq \{1,\dots,n-2\}$ are realized as the delta set of an individual element in $M$ (i.e., for which $T$ does there exist an $x \in M$ such that $T = \Delta(x))$? Based on the structure theorem for sets of lengths in Krull monoids with finite class group (see \cite[Chapter~4]{GH06}), it is reasonable to assume that not all subsets of $\{1, \dots, n-2\}$ will be realized. We verify this in Theorem~\ref{thm:first main result} by showing for $x \in M$ that if $n-2 \in \Delta(x)$, then $\Delta(x) = \{n-2\}$. We contrast this in Theorem~\ref{thm:Archimedean property} by showing that we can construct delta sets of arbitrarily large size (i.e., for any $m \in \nn$ there exists a Krull monoid $M$ with finite cyclic class group such that $M$ has an element $x$ with $|\Delta(x)| \ge m$).
	
	\section{Definitions and Background} \label{sec:definitions and background}
	
	We open with some basic definitions from the theory of nonunique factorizations. For a commutative cancellative monoid $M$, let $\mathcal{A}(M)$ represent the set of \emph{atoms} (i.e., irreducible elements) of $M$, and $M^\times$ its set of units. We provide here an informal description of factorizations and associated notions. From a more formal point of view, factorizations can be considered as elements in the \emph{factorization monoid} $\mathsf{Z}(M)$, which is defined as the free abelian monoid with basis $\mathcal{A}(M/M^\times)$; details can be found in the first chapter of \cite{GH06}.
	
	To simplify our initial discussion, we suppose that $M$ is reduced (i.e., has a unique unit). We say that $a_1 \dots a_k$ is a \emph{factorization} of $x \in M$ if $a_1, \dots, a_k \in \mathcal{A}(M)$ and $x = a_1 \dots a_k$ in $M$. Two factorizations are equivalent if there is a permutation of atoms carrying one factorization to the other. We denote by $\mathsf{Z}(x) \subseteq \mathsf{Z}(M)$ the set of all factorizations of $x$. If $z \in \mathsf{Z}(x)$, then let $|z|$ denote the number of atoms in the factorization $z$ of $x$. We call $|z|$ the \emph{length} of $z$. Now, let $x \in M \setminus M^\times$ with factorizations
	\[
		z = \alpha_1 \cdots \alpha_t \beta_1 \cdots \beta_s \ \text{ and } \ z' = \alpha_1 \cdots \alpha_t \gamma_1 \cdots \gamma_u
	\]
	where for each $1 \le i \le s$ and $1 \le j \le u$, $\beta_i \neq \gamma_j$. Define
	\[
		\gcd(z, z') = \alpha_1 \cdots \alpha_t
	\]
	and
	\[
		\mathsf{d}(z, z') = \max\{s, u\}
	\]
	to be the \emph{distance} between $z$ and $z'$. The basic properties of this distance function can be found in \cite[Proposition~1.2.5]{GH06}.
	
	An $N$-\emph{chain of factorizations} from $z$ to $z'$ is a sequence $z_0, \dots, z_k$ such that each $z_i$ is a factorization of $x$, $z_0 = z$, $z_k = z'$ and $\mathsf{d}(z_i, z_{i+1}) \le N$ for all $i$. The \emph{catenary degree} of $x$, denoted $\mathsf{c}(x)$, is the minimal $N \in \nn_0 \cup \{\infty\}$ such that for any two factorizations $z, z'$ of $x$, there is an $N$-chain from $z$ to $z'$. The \emph{catenary degree} of $M$, denoted by $\mathsf{c}(M)$, is defined by
	\[
		\mathsf{c}(M) = \sup\{\mathsf{c}(x) \mid x \in M \setminus M^\times\}.
	\]
	A review of the known facts concerning the catenary degree can be found in \cite[Chapter~3]{GH06}. An algorithm which computes the catenary degree of a finitely generated monoid can be found in \cite{CGLPR06}, and a more specific version for numerical monoids in \cite{CGL09}.
	
	We shift from considering particular factorizations to analyzing their lengths. For $x \in M \setminus M^\times$, we define
	\[
		\mathsf{L}(x) = \{n \in \nn \mid \text{ there are } \ \alpha_1, \dots, \alpha_n \in \mathcal{A}(M) \ \text{ with } \ x = \alpha_1 \cdots \alpha_n\}.
	\]
	We refer to $\mathsf{L}(x)$ as the \emph{set of lengths} of $x$ in $M$. Further, set
	\[
		\mathcal{L}(M) = \{\mathsf{L}(x) \mid x \in M \setminus M^\times\},
	\]
	which we refer to as the \emph{system of sets of lengths} of $M$. The interested reader can find many recent advances concerning sets of lengths in \cite{GG00}, \cite{wS09}, and \cite{wS09a}. Given $x \in M \setminus M^\times$, write its length set in the form
	\[
		\mathsf{L}(x) = \{n_1, \dots, n_k\},
	\]
	where $n_i < n_{i+1}$ for $1 \le i \le k-1$. The \emph{delta set} of $x$ is defined by
	\[
		\Delta(x) = \{n_i - n_{i-1} \mid 2 \le i \le k\},
	\]
	and the \emph{delta set} of $M$ (also called the \emph{set of distances} of $M$) by
	\[
		\Delta(M) = \bigcup_{x \in M \setminus M^\times} \Delta(x)
	\]
	(see again \cite[Chapter ~1.4]{GH06}). Computations of delta sets in various types of monoids can be found in \cite{BCS08}-\cite{CCS07}.
	
	A monoid $M$ is called a \emph{Krull monoid} if there is a monoid homomorphism $\varphi \colon M \to D$ where $D$ is a free abelian monoid and $\varphi$ satisfies the following two conditions:
	\begin{enumerate}
		\item if $a, b \in M$ and $\varphi(a) \mid \varphi(b)$ in $D$, then $a \mid b$ in $M$,
		\item for every $\alpha \in D$ there exist $a_1, \dots, a_n \in M$ with
		\[
			\alpha = \gcd\{\varphi(a_1), \dots, \varphi(a_n)\}.
		\]
	\end{enumerate}
	Clearly, a monoid $M$ is Krull if and only if the associated reduced monoid is Krull. The basis elements of $D$ are called the prime divisors of $M$. The above properties guarantee that Cl$(M) = D/\varphi(M)$ is an abelian group, which we call the \emph{class group} of $M$ (see \cite[Section~2.3]{GH06}). Note that since any Krull monoid is isomorphic to a submonoid of a free abelian monoid, a Krull monoid is commutative, cancellative, and atomic. The class of Krull monoids contains many well-studied types of monoids, such as the multiplicative monoid of a ring of algebraic integers (see \cite{BC11,aG09,GH06}).
	
	Let $G$ be an abelian group and $\mathcal{F}(G)$ the free abelian monoid on $G$. The elements of $\mathcal{F}(G)$, which we write in the form
	\[
		X = g_1 \dots g_l = \prod_{g \in G} g^{\mathsf{v}_g(X)},
	\]
	are called \emph{sequences} over $G$. We set $-X = (-g_1) \cdots (-g_l)$. The exponent $\mathsf{v}_g(X)$ is the \emph{multiplicity} of $g$ in $X$. The \emph{length} of $X$ is defined as
	\[
		|X| = l =\sum_{g \in G} \mathsf{v}_g(X)
	\]
	(note that as $\mathcal{F}(G)$ and $\mathsf{Z}(x)$ are both free abelian monoids; there is really no redundancy in this notation). For every $I \subseteq [1, l]$, the sequence $Y = \prod_{i \in I} g_i$ is called a \emph{subsequence} of $X$. The subsequences are precisely the divisors of $X$ in the free abelian monoid $\mathcal{F}(G)$. The submonoid
	\[
		\mathcal{B}(G) = \bigg\{ X \in \mathcal{F}(G) \ \bigg{|} \ \sum_{g \in G} \mathsf{v}_g(X) g = 0 \bigg\}
	\]
	is known as the \emph{block monoid} on $G$, and its elements are referred to as \emph{zero-sum sequences} or \emph{blocks} over $G$ (\cite[Section~2.5]{GH06} is a good general reference on block monoids). If $S$ is a subset of $G$, then the submonoid
	\[
		\mathcal{B}(G, S) = \{ X \in \mathcal{B}(G) \mid \mathsf{v}_g(X) = 0 \ \emph{ if } \ g \notin S \}
	\]
	of $\mathcal{B}(G)$ is called the \emph{restriction} of $\mathcal{B}(G)$ to $S$. Block monoids are important examples of Krull monoids and their true relevance in the theory of nonunique factorizations lies in the following result.
	
	\begin{prop} \label{prop:relation of Krull monoids and block monoids}\cite[Theorem~3.4.10.3]{GH06}
		Let $M$ be a Krull monoid with class group $G$ and let $S$ be the set of classes of $G$ which contain prime divisors. Then
		\[
			\mathcal{L}(M) = \mathcal{L}(\mathcal{B}(G, S)).
		\]
	\end{prop}
	
	Hence, to understand the arithmetic of lengths of factorizations in a Krull monoid, one merely needs to understand the factorization theory of block monoids. Thus, while we state our results in the context of general Krull monoids, Proposition~\ref{prop:relation of Krull monoids and block monoids} allows us to write the proofs using block monoids.
	
	For our purposes, there are two arithmetic properties of the block monoid $\mathcal{B}(G)$, where $G$ is cyclic of order $n \ge 3$, which we will use later:
	\begin{enumerate}
		\item $\mathsf{c}(\mathcal{B}(G)) = n$ (see \cite[Theorem~6.4.7]{GH06});
		\item $\Delta(\mathcal{B}(G)) = \{1, \dots, \mathsf{c}(\mathcal{B}(G))-2\}$ (see \cite[Theorem~6.7.1.4]{GH06}).
	\end{enumerate}
	Thus $\Delta(\mathcal{B}(G)) = \{1, \dots, n-2\}$. For any finite abelian group $G$ with $|G| \ge 3$, the set $\Delta(\mathcal{B}(G))$ is an interval whose minimum equals $1$ but whose maximum is not
	known in general \cite{GGS11,GY12}. We will be interested in the following types of subsets of $\Delta(M)$.
	
	\begin{definition}
		Let $M$ be a commutative cancellative monoid and suppose $T$ is a nonempty subset of $\Delta(M)$. We call $T$ \emph{realizable} in $\Delta(M)$ if there is an element $x \in M$ with $\Delta(x) = T$.
	\end{definition}
	
	\begin{example} \label{ex:realizable subsets}
		We illustrate some aspects of the last definition with several examples.
		
		\begin{enumerate}
			
			\item Let $M$ be any Krull monoid $M$ with $\Delta(M) = \{c\}$ for $c \in \nn$. Clearly any element $x \in M$ with $|\mathsf{L}(x)| > 1$ yields $\Delta(x) = \{c\}$ and hence every subset of $\Delta(M)$ is realizable. A large class of Dedekind domains (whose multiplicative monoids are Krull) with such delta sets are constructed in \cite{CS87}. Another example of such a monoid is a primitive numerical monoid whose minimal generating set forms an arithmetic sequence (see \cite[Theorem~3.9]{BCKR06}).
			
			\item Let $G$ be any infinite abelian group and $S$ any finite nonempty subset of $\{2, 3, 4, \dots\}$. By a well-known theorem of Kainrath (\cite{fK99}, \cite[Section~7.4]{GH06}), there is a block $B \in \mathcal{B}(G)$ with $\mathsf{L}(B) = S$. From this, it easily follows that $\Delta(\mathcal{B}(G)) = \nn$. Moreover, it also easily follows that any finite subset $T$ of $\nn$ is realizable in $\Delta(\mathcal{B}(G))$. In this example we can actually say more. If $H$ is a monoid, then recall that a submonoid $S \subset H$ is \emph{divisor-closed} if for for all $a \in S$ and $b \in H$ one has that $b \mid_H a$ implies that $b \in S$. Set
			\[
				\phantom{b-space} \Delta^*(H) = \{\min \Delta(S) \mid S \subset H \ \text{is a divisor-closed submonoid and } \Delta(S) \neq \emptyset\}.
			\]
			It follows from \cite{CSS08} that $\Delta^*(\mathcal{B}(G)) = \nn$.
			
			\item In general, there are commutative cancellative monoids $M$ with unrealizable subsets of $\Delta(M)$. For our initial example, we again appeal to numerical monoids. Let $S$ be the monoid of positive integers under addition generated by $4$, $6$, and $15$ (i.e., $S = \langle 4, 6, 15 \rangle$). By \cite[Example~2.6]{BCKR06}, it follows that $\Delta(S) = \{1, 2, 3\}$. By \cite[Theorem~1]{CHK09}, the sequence of sets $\{\Delta(n)\}_{n \in S}$ is eventually periodic, and hence using the periodic bound in that theorem allows one to check for all realizable subsets of $\{1, 2, 3\}$ in finite time. Using programming from the GAP NumericalSgps package \cite{DGM13}, one can verify that $\{1\}$, $\{1, 2\}$, and $\{1, 3\}$ are the only realizable subsets of $\Delta(S)$.
			
			\item There are three generated numerical monoids that behave differently than the two types discussed above. For instance, let $S$ be the numerical monoid generated by $7$, $10$, and $12$ (i.e., $S = \langle 7, 10, 12 \rangle$). By \cite[Example~2.5]{BCKR06}, one has $\Delta(S) = \{1, 2\}$. Again, using the GAP programming \cite{DGM13}, we find that $\Delta(34) = \{1\}$, $\Delta(42) = \{2\}$, and $\Delta(56) = \{1, 2\}$. Thus, all nonempty subsets of $\Delta(S)$ are realizable.
			
			\item If $G$ is cyclic of order $n \ge 3$, then by \cite[Corollary~2.3.5]{aG09}, for each $1 \le i \le n-2$, the set $\{i\}$ is realizable in $\Delta(\mathcal{B}(G))$.
			
			\item Suppose $G$ is an abelian group and $S_1 \subseteq S_2$ are subsets of $G$. By the properties of the block monoid, if $B \in \mathcal{B}(G, S_1)$, then $\mathsf{L}(B)$ is equal in both $\mathcal{B}(G, S_1)$ and $\mathcal{B}(G, S_2)$. Thus, if $T$ is realizable in $\Delta(\mathcal{B}(G, S_1))$, then $T$ is realizable in $\Delta(\mathcal{B}(G, S_2))$. Elementary examples show that this relationship does not work conversely.
			
		\end{enumerate}
	\end{example}
	
	Our eventual goal is to show that in contrast to $\mathcal{B}(G)$ where $G$ is infinite abelian, not all subsets of $\Delta(\mathcal{B}(G))$ are realizable when $G$ is finite cyclic.
	
	\section{Main Results} \label{sec:main results}
	
	Our first lemma will be vital in the proof of Theorem~\ref{thm:first main result}.
	
	\begin{lemma} \label{lem:vital lemma}
		Let $G$ be a cyclic group of order $|G| = n \ge 3$, $g \in G$ with \emph{ord}$(g) = n$, $V = (-g)g$, and $W = g^n$. Let $u \in \mathcal{B}(G)$ and $z, z' \in \mathsf{Z}(u)$, say
		\[
			z = W^r(-W)^sBV^q, \ \text{ where } \ r, s > 0, q \ge 0,
		\]
		and $B$ is the product of atoms which are not in the set $\{W, -W, V \}$. If $|z'|-|z| = n-2$ and there are no factorizations of $u$ having length between $|z|$ and $|z'|$, then $B = 0^t$ for some $t \in \nn_0$.
	\end{lemma}
	
	\begin{proof}
		Since $0 \in \mathcal{B}(G)$ is a prime element and $\mathsf{L}(u) = \mathsf{v}_0(u) + \mathsf{L}(0^{-\mathsf{v}_0(u)}u)$, we may assume that $0 \nmid u$, and we have to show that $B = 1$ (i.e., $B$ is the empty block). This is obvious for $n = 3$, and hence we suppose that $n \ge 4$. Assume, by way of contradiction, that $A$ is an atom of $\mathcal{B}(G)$ dividing $B$. Since $0 \nmid B$, it follows that $|A| \ge 2$.
		
		Suppose that $|A| = 2$, say $A = (xg)(-xg)$ with $x \in \{2, \dots, \lfloor n/2 \rfloor\}$. Then
		\[
			W A(-W) = (g^x(-xg))((-g)^x(xg)V^{n - x}
		\]
		has a factorization of length $2+n-x$. Since $3 < 2 + n - x < n + 1$, we obtain a factorization of $u$ with length strictly between $|z|$ and $|z'|$, a contradiction.
		
		Suppose now that $|A| \ge 3$, say $A = (xg)(yg)(wg)A'$ with $x, y, w \in \{1, \dots, n - 1\}$ and $A' \in \mathcal{F}(G)$. Then two of the three elements $-xg$, $-yg$, $-wg$ are either in the set $C = \{g, 2g, \dots, \lfloor n/2 \rfloor g\}$ or in the set $D = \{(\lfloor n/2 \rfloor + 1)g, \dots, (n - 1)g\}$. After renaming $xg$, $yg$ and $wg$ and exchanging $g$ and $-g$ if necessary, we may suppose that $-xg = (n - x)g \in C$ and $-yg = (n - y)g \in C$. Then WA is divisible by the product of two atoms $((xg)g^{n-x})((yg)g^{n-y})$. If $(n - x) + (n - y) = n$, then $n$ is even, $x = y = n/2$, and $(xg)(yg)$ is a zero-sum subsequence of $A$, a contradiction. Thus $(n - x) + (n - y) < n$ and WA is the product of at least three atoms. Since $A \neq -W$, it follows that $AW$ is a product of at most $n-1$ atoms, and thus we obtain a factorization of $u$ with length strictly between $|z|$ and $|z'|$, a contradiction.
	\end{proof}
	
	Lemma~\ref{lem:vital lemma} leads us to our first main result.
	
	\begin{theorem} \label{thm:first main result}
		Let $M$ be a Krull monoid with cyclic class group $G$ of order $n \ge 3$. If $x \in M$ and $n - 2 \in \Delta(x)$, then $\Delta(x) = \{n - 2\}$.
	\end{theorem}
	
	\begin{proof}
		By Proposition~\ref{prop:relation of Krull monoids and block monoids}, we need only prove the theorem for $\mathcal{B}(G, S)$ where $S \subseteq G$ is the set of classes which contain prime divisors. By our comment in Example~\ref{ex:realizable subsets}(6), it suffices to prove our theorem for $S = G$ (i.e., for $\mathcal{B}(G)$).
		
		Let $x \in \mathcal{B}(G)$ and $z, z' \in \mathsf{Z}(x)$ be such that $|z'| - |z| = n - 2$ and there are no factorizations of $x$ with length between $|z|$ and $|z'|$. As in Lemma~\ref{lem:vital lemma} we may suppose that $0 \nmid x$. We shall prove that there exists an element $g \in G$ such that $z$ is divisible by the atoms $g^n$ and $(-g)^n$. By a previous observation, we know that the catenary degree of $\mathcal{B}(G)$ is $n$. Therefore, $\mathsf{c}(x) \le n$ and so there exists an $n$-chain $z = z_0, z_1, \dots, z_k = z'$ of factorizations of $x$ from $z$ to $z'$. We also know, by \cite[Lemma~1.6.2]{GH06}, that
		\[
			\big{|}|z_{i+1}| - |z_i|\big{|} \le \mathsf{d}(z_i, z_{i+1}) - 2 \le n - 2.
		\]
		Let $j$ be the least index such that $|z_{j+1}| > |z|$. Then $|z_j| = |z|$ and $|z_{j+1}| = |z'|$ because there are no factorizations of $x$ with length between $|z|$ and $|z'|$ and also because $\big{|} |z_{i+1}| - |z_i|\big{|} \le n - 2$. It follows that $|z_{j+1}| - |z_j| = |z'| - |z| = n - 2$, and so
		\[
			\mathsf{d}(z_j, z_{j+1}) \ge \big{|}|z_{j+1}| - |z_j|\big{|} + 2 = n.
		\]
		Therefore, by redefining, if necessary, $z$ and $z'$ as $z_j$ and $z_{j+1}$ respectively, we can assume that $\mathsf{d}(z, z') = n$.
	
		Let $d$ be the greatest common divisor of $z$ and $z'$ in the factorization monoid $\mathsf{Z}(\mathcal{B}(G))$, and take $w$ and $w'$ so that $z = dw$ and $z' = dw'$. Notice that $|w'| - |w| = |z'| - |z| = n - 2$. As $\max\{|w|,|w'|\} = \mathsf{d}(z, z') = n$, we have $|w'| = n$ and $|w| = 2$; note that so far $|\cdot|$ referred to the length in the factorization monoid $\mathsf{Z}(\mathcal{B}(G)) = \mathcal{F}(\mathcal{A}(\mathcal{B}(G)))$.
		
		It is well known that $U \in \mathcal{A}(\mathcal{B}(G))$ implies that $|U| \le n$, and $|U| = n$ if and only if $U = g^n$ for some $g \in G$ (again by \cite[Theorem~5.1.10]{GH06}; now $|\cdot|$ refers to $\mathcal{F}(G)$). Since $w$ consists of only two atoms, say $w = U_1U_2$ with $U_1, U_2 \in \mathcal{A}(\mathcal{B}(G))$, we obtain $|U_1U_2| \le 2n$. Since $w'$ is a product of exactly $n$ atoms, say $w' = V_1 \cdots V_n$ with $V_1, \dots, V_n \in \mathcal{A}(\mathcal{B}(G))$, we infer that $|V_1| = \dots = |V_n| = 2$. Then $|U_1 U_2| = 2n$, and since $w'$ is divisible by an atom of length $2$, we see that $U_1 = W = g^n$ and $U_2 = -W = (-g)^n$ for some element $g \in G$ of order $n$.
		
		Now we can write $z = W^r(-W)^sBV^q$ where $V = (-g)g$ and $B$ is not divisible by $W$, $-W$, or $V$. Since there are no factorizations of $x$ having length between $|z|$ and $|z'|$, Lemma~\ref{lem:vital lemma} implies that $B = 1 \in \mathsf{Z}(\mathcal{B}(G))$ and $z = W^r(-W)^sV^q$. Now it easily follows that $\Delta(x) = \{n - 2\}$ (see \cite[Proposition~4.1.2]{GH06}).
	\end{proof}
	
	As an immediate consequence of Theorem~\ref{thm:first main result}, we have the following result.
	
	\begin{cor}
		Let $M$ be a Krull monoid with cyclic class group $G$ of order $n \ge 3$. If T is a nonempty subset of $\{1, \dots, n-2\}$ with $n-2 \in T$ but $T \neq \{n-2\}$, then $T$ is not realizable in $\Delta(M)$.
	\end{cor}
	
	\begin{example}
		Let $G$ be a cyclic group of order $|G| = 5$ and $g \in G$ an element of order $5$. We use the last two results to determine all the realizable sets of $\Delta(\mathcal{B}(G))$. Since $\Delta(\mathcal{B}(G)) = \{1, 2, 3\}$, Theorem~\ref{thm:first main result} implies that the only possible realizable sets are $\{1\}$, $\{2\}$, $\{3\}$, and $\{1, 2\}$. The singleton sets are guaranteed by Example~\ref{ex:realizable subsets}(5). Let $B = g^8(2g)(-g)^5$. We claim that $\Delta(B) = \{1, 2\}$. If $A$ is an atom with $(2g) \mid A \mid B$, we deduce that $A = (2g)g^3$ or $A = (2g)(-g)^2$. The fact that $\mathsf{L}(g^5(-g)^5) = \{2, 5\}$ implies the existence of factorizations of $B$ with lengths $3$ and $6$ when $A = (2g)g^3$. Observe also that having $A = (2g)(-g)^2$ determines uniquely the factorization of $B$ given by $z = ((2g)(-g)^2)(g(-g))^3(g^5) \in \mathsf{Z}(B)$. As a result, $\mathsf{L}(B) = \{3, 5, 6\}$ and, therefore, $\Delta(B) = \{1, 2\}$. Hence, $\{1\}$, $\{2\}$, $\{3\}$, and $\{1, 2\}$ determine the complete set of realizable sets of $\Delta(\mathcal{B}(G))$.
	\end{example}
	
	Our results to this point have produced delta sets of relatively small size. However, we now prove that we can have realizable sets of large size provided we choose an adequate finite cyclic group.
	
	\begin{theorem} \label{thm:Archimedean property}
		For every $m \in \nn$ there exist a Krull monoid M with finite cyclic class group and an element $x \in M$ such that $|\Delta(x)| \ge m$.
	\end{theorem}
	
	\begin{proof}
		Choose a natural number $b_0 > m$ and define $b_1, \dots, b_m$ recursively by setting $b_{k+1} = 2(\sum_{i=0}^k b_i) + 2m$ if $k > 0$. Let $B = b_1 \cdots b_m \in \mathcal{F}(\zz)$ be the sequence over $\zz$, $\sigma(B)$ its sum, and
		\[
			\Sigma(B) = \bigg\{ \sum_{i \in I} b_i \ \bigg{|} \ \emptyset \neq I \subseteq \{1,\dots,m\} \bigg\}
		\]
		its set of subsequence sums.
		
		Choose a finite cyclic group $G$ of order $|G| = n > \sigma(B)$ and set $M =\mathcal{B}(G)$. Let $g \in G$ with ord$(g) = n$, and consider the element
		\[
			x = g^{2n - \sigma(B)}(-g)^n \prod_{i=1}^m (b_ig) \in \mathcal{B}(G).
		\]
		For every subset $I \subseteq \{1, \dots, m\}$, the element $A_I = g^{n - \sum_{i \in I} b_i}\prod_{i \in I} (b_ig)$ is an atom of $\mathcal{B}(G)$. Suppose there are subsets $I, J \subseteq\{1, \dots, m\}$ such that $A_I A_J \mid x$. Then we get $I \cap J = \emptyset$, $I \uplus J \subseteq \{1, \dots, m\}$, and
		\[
			\bigg(n - \sum_{i \in I} b_i \bigg) + \bigg(n - \sum_{j \in J} b_j \bigg) = \mathsf{v}_g(A_I A_J) \le \mathsf{v}_g(x) = 2n - \sigma(B).
		\]
		Therefore $I \uplus J = \{1, \dots, m\}$ and $x = A_I A_J(-g)^n$. Let $z$ be a factorization of $x$. Then there is a subset $I \subseteq \{1, \dots, m\}$ such that $A_I \mid z$. If $A_J \nmid z$ for $J = \{1, \dots, m\} \setminus I$, then
		\[
			z = A_I \prod_{j \in J} ((b_jg)(-g)^{b_j})((-g)g)^{n - \sum_{j \in J} b_j}
		\]
		and $|z| = 1 + |J| + n - \sum_{j \in J} b_j$. Therefore
		\[
			\mathsf{L}(x) = \{3\} \cup \bigg\{ 1 + |J| + n - \sum_{j \in J} b_j \ \bigg{|} \ J \subseteq \{1, \dots, m\} \bigg\}.
		\]
		
		In order to point out that there are $m$ distinct elements in $\Delta(x)$, let
		\[
			L_k = \bigg\{ 1 + |J| + n - \sum_{j \in J} b_j \ \bigg{|} \ J \subseteq \{1, \dots, m\} \ \text{ with } \ \max J = k \bigg\}
		\]
		for each $k \in \{0, \dots, m\}$, with the usual convention that
		\[
			L_0 = \bigg\{ 1 + |\emptyset| + n - \sum_{j \in \emptyset} b_j = 1 + n \bigg\}.
		\]
		If $k \in \{1, \dots, m\}$, then
		\[
			\min L_{k-1} = 1 + (k - 1) + n - \sum_{j=1}^{k-1} b_j > 1 + 1 + n - b_k = \max L_k,
		\]
		and so
		\[
				\mathsf{L}(x) = \{3\} \cup \biguplus_{k=1}^m L_k.
		\]
		Hence, for each $k \in \{1, \dots, m\}$, the element $\min L_{k-1} - \max L_k = b_k - \sum_{j=1}^{k-1} b_j + (k - 2)$ is in $\Delta(x)$, and the growth condition on the elements $b_1, \dots, b_m$ guarantees that these values are pairwise distinct.
	\end{proof}
	
	Let $G$ be a finite abelian group with $|G| \ge 3$. By Theorem~\ref{thm:first main result}, cyclic groups have the following property: If $\max \Delta(\mathcal{B}(G)) - 2 \in \Delta(B)$ for some $B \in \mathcal{B}(G)$, then
	\[
		\Delta(B) = \{\max \Delta(\mathcal{B}(G)) - 2 \}.
	\]
	It was recently shown in \cite{GS14} that elementary $2$-groups do share this property, and it is a challenging problem to characterize all such groups. A further wide open question is to study
	\[
		\Lambda(G) = \max\{|\Delta(B)| \mid B \in \mathcal{B}(G)\}.
	\]
	If $G$ is cyclic of order $n \ge 13$, we have shown that $4 \le \Lambda(G) \le n - 3$.
	
	\section{Acknowledgments}
	
	It is a pleasure to thank the referee for valuable suggestions which resulted in an improvement of the manuscript. The authors were supported by National Science Foundation grants DMS-1035147 and DMS-1045082 and a supplemental grant from the National Security Agency.

\end{document}